\documentclass[12pt,a4paper]{article}%
\usepackage{amsmath, amsfonts, amsthm,color, latexsym, amssymb}
\usepackage{amscd}
\usepackage{amsmath}
\usepackage{amsfonts}
\usepackage{amssymb}
\usepackage{graphicx}%
\providecommand{\U}[1]{\protect\rule{.1in}{.1in}}
\newtheorem{theorem}{Theorem}[section]
\newtheorem{proposition}[theorem]{Proposition}
\newtheorem{corollary}[theorem]{Corollary}
\newtheorem{example}[theorem]{Example}

\newtheorem{remark}[theorem]{Remark}

\newtheorem{lemma}[theorem]{Lemma}
\newtheorem{definition}[theorem]{Definition}
\begin{document}

\title{\textsc{Approximation properties determined by operator ideals}}
\author{Sonia Berrios and Geraldo Botelho\thanks{Supported by CNPq Grant 306981/2008-4.\hfill\newline2010
}}
\date{}
\maketitle

\begin{abstract}
Given an operator ideal $\cal I$, a Banach space $E$ has the $\cal I$-approximation property if operators on $E$ can be uniformly approximated on compact subsets of $E$ by operators belonging to $\cal I$. In this paper the $\cal I$-approximation property is studied in projective tensor products, spaces of linear functionals, spaces of homogeneous polynomials (in particular, spaces of linear operators), spaces of holomorphic functions and their preduals.
\end{abstract}


\section{Introduction}
Given Banach spaces $E$ and $F$, by ${\cal L}(E;F)$ we denote the Banach space of all bounded linear operators
from $E$ to $F$ endowed with the usual operator sup norm.
The subspaces of ${\cal L}(E;F)$ formed by all finite rank, compact and weakly compact operators
 are denoted by ${\cal F}(E;F)$, ${\cal K}(E;F)$ and ${\cal W}(E;F)$, respectively. For a subset $S$ of ${\cal L}(E;F)$,
  the symbol $\overline{S}^{\tau_c}$ represents the closure of $S$ with respect to the compact-open topology $\tau_c$.\\
\indent It is well known that a Banach space $E$ has\\
$\bullet$ the approximation property (in short, $E$ has AP) if ${\cal L}(E;E) = \overline{{\cal F}(E;E)}^{\tau_c}$,\\
$\bullet$ the compact approximation property (in short, $E$ has CAP) if ${\cal L}(E;E) = \overline{{\cal K}(E;E)}^{\tau_c}$,\\
$\bullet$ the weakly compact approximation property (in short, $E$ has WCAP) if ${\cal L}(E;E) = \overline{{\cal W}(E;E)}^{\tau_c}$.\\
\indent The AP is a classic in Banach space theory (see \cite{Casazza}) and is one of the main subjects of Grothendieck \cite{Grothendieck}.
 The CAP has been more studied in the last decades and recently (see, e.g. \cite{choi-kim, choi-kim-outro, Caliskan3, Caliskan1}), but it goes back to Banach \cite[p.\,237]{banach}. The WCAP has been studied more recently (see \cite{Caliskan1,Caliskan4}).
 Having in mind that ${\cal F}, {\cal K}$ and $\cal W$ are operator ideals, the properties above
 can be regarded as particular instances of the following general concept:

\begin{definition} \label{def}\rm Let $\cal I$ be an operator ideal. A Banach space $E$ is said to have the {\rm $\cal I$-approximation property} (in short, $E$ has $\cal I$-AP) if ${\cal L}(E;E) = \overline{{\cal I}(E;E)}^{\tau_c}$.
\end{definition}
Several variants of the approximation property have been studied recently (see, e.g, \cite{choi-kim, choi-kim-lee, dops, delgadoJMAA, lo, oja2008, sinha}), including ones closely related to the $\cal I$-AP \cite{llo, lmo, oja}.\\
\indent The selection of operator ideals instead of other classes of linear operators related to $\cal F$, $\cal K$ and $\cal W$ is justified by the fact that even the most basic results depend on the ideal property (cf. Section 3).\\
\indent It is clear that if $E$ has AP then $E$ has $\cal I$-AP for every operator ideal $\cal I$.
In particular, Banach spaces with Schauder basis (e.g., $\ell_p$, $1 \leq p < \infty$, and $c_0$)
have $\cal I$-AP for every operator ideal $\cal I$. \\
\indent Let us stress that different ideals may give rise to different approximation properties: (i) Willis \cite{Willis} showed that there are spaces with CAP
but not with AP; (ii) Szankowski \cite{szankowski} proved that for $1 \leq p < 2$, $\ell_p$ has a subspace $S_p$ without CAP, so $S_{\frac32}$ has WCAP but not CAP and $S_1$ has ${\cal CC} \cap {\cal C}_2$-AP but not CAP, where ${\cal CC}$ and ${\cal C}_2$ are the ideals of completely continuous and cotype 2 operators, respectively. The fact that different operator ideals usually give rise to different approximation properties justifies the study of the $\cal I$-AP for arbitrary operator ideals, which is the aim of this paper. In Example \ref{example} we shall see that different ideals may generate the same approximation property.\\
\indent The study of the approximation property and its already studied variants is very rich and multifaceted, so the study of the $\cal I$-AP could follow several different trends. This means that, to study the $\cal I$-AP, choices have to be made. In this paper we have chosen to study the $\cal I$-AP in projective tensor products (Section 5) and in spaces of mappings between Banach spaces, namely, spaces of linear functionals (Section 4), spaces of homogeneous polynomials (Section 6) and spaces of holomorphic functions and their preduals (Section 7). Proposition \ref{certo} fixes and generalizes a result of \cite{Caliskan4}.\\
 \indent The results we prove in the different sections of the paper seem - at first glance - to be completely disconnected. Connections of results from different sections are given in Section 7.

\section{ Notation and preliminaries}
When $F$ is the scalar field $\mathbb{K} = \mathbb{R}$ or $\mathbb{C}$, we shall write $E'$ instead of $\mathcal{L}(E;\mathbb{K})$.
 The  {\it compact-open topology or topology of compact convergence} is the locally convex topology $\tau_c$
on  $\mathcal{L}(E;F)$ which is generated by the seminorms of the form
$$p_K(T)=\sup_{x\in K}\|T(x)\|,$$
where $K$ ranges over all compact subsets of  $E$.\\
\indent Given a subset $S$ of $\mathcal{L}(E;F)$, $\overline{S}^{\tau_c} = \mathcal{L}(E;F)$ if and only if for every $T \in \mathcal{L}(E;F)$, every compact set $K\subseteq E$ and every $\varepsilon >0$, there is an operator $U\in S$ such that $\|T(x)-U(x)\|<\varepsilon$ for every $x\in K$.

\indent An {\it operator ideal} $\mathcal{I}$ is a subclass of the class
of all continuous linear operators between Banach spaces such that
for all Banach spaces  $E$ e $F$, the component
$\mathcal{I}(E;F)=\mathcal{L}(E;F)\cap \mathcal{I}$ satisfy:\\
(a) $\mathcal{I}(E;F)$ is a linear subspace of
$\mathcal{L}(E;F)$ which contains the finite rank operators.\\
(b) Ideal property: If $T\in \mathcal{L}(E;F)$, $R\in
\mathcal{I}(F;G)$ and $S\in \mathcal{L}(G;H)$, then the composition
 $S\circ R\circ T$ is in $\mathcal{I}(E;H)$.

\medskip

 By $id_E$ we mean the identity operator on the Banach space $E$. For a given operator ideal $\cal I$, by $\overline{\cal I}$ we mean the closure of $\cal I$, that is, $\overline{\cal I}(E;F) = \overline{{\cal I}(E;F)}$ for every Banach spaces $E$ and $F$. For the theory of operator ideals we refer to \cite{Pietsch,klaus}. Here is a list of the operator ideals occurring in this paper:\\
$\cal F$ = finite rank operators (the range is finite-dimensional),\\
$\cal A := \overline{\cal F}$ = approximable operators,\\
$\cal K$ = compact operators (bounded sets are mapped onto relatively compact sets),\\
$\cal W$ = weakly compact operators (bounded sets are mapped onto relatively weakly compact sets),\\
${\cal CC}$ = completely continuous operators (weakly convergent sequences are sent to norm convergent sequences),\\
${\cal N}_p$ = $p$-nuclear operators,\\
${\cal C}_p$ = cotype $p$ operators,\\
${\cal T}_p$ = type $p$ operators,\\
$\cal D$ = dualisable operators,\\
$\cal S$ = separable operators (the range is separable),\\
$\cal DP := \mathcal{W}^{-{\rm 1}} \circ {\cal CC}$ = Dunford-Pettis operators,\\
$\cal J$ = integral operators,\\
$\cal SN$ = strongly nuclear operators,\\
$\cal SS$ = strictly singular operators (restrictions to infinite-dimensional subspaces are never isomorphisms),\\
$\cal SC$ = strictly cosingular operators,\\
$\Pi_p$ = absolutely $p$-summing operators,\\
$\Pi_{r,p,q}$ = absolutely $(r,p,q)$-summing operators \cite[17.1]{Pietsch},\\
$\Gamma_p$ = $p$-factorable operators,\\
$\cal KC$ = K-convex operators \cite[31.1]{klaus},\\
$\cal QN$ = quasinuclear operators \cite[Ex. 9.13]{klaus},\\
${\cal L}_{\infty, q, \gamma}$ = Lorentz-Zigmund operators \cite{cobos},\\
${\cal U}_p $ = operators having approximation numbers belonging to $\ell_p$  \cite[14.2.4]{Pietsch},\\
${\cal L}_{p,q}$ = $(p,q)$-factorable operators,\\
${\cal K}_p$ = $p$-compact operators (bounded sets are mapped to relatively $p$-compact sets),\\
${\cal QN}_p$ = quasi $p$-nuclear operators \cite{delgadostudia}.\\


\section{Basic results}
The results of this section, except for some implications of Proposition \ref{charac}, are elementary enough to have their proofs omitted. Nevertheless, it is worth mentioning that the ideal property plays a crucial role in their (easy) proofs.

The following characterizations are simple but useful.

\begin{proposition}\label{propos} Given an operator ideal $\cal I$, the following are equivalent for a Banach space $E$:\\
{\rm (a)} $E$ has the $\cal I$-approximation property.\\
{\rm (b)} $E$ has the $\overline{\cal I}$-approximation property.\\
{\rm (c)} $id_E \in \overline{{\cal I}(E;E)}^{\tau_c}$.\\
{\rm (d)} For each compact set $K\subseteq
E$ and every $\varepsilon
>0$, there is an operator $T\in \mathcal{I}(E;E)$ such that
$\|T(x)-x\|<\varepsilon$ for every $x\in K$.
\end{proposition}

Given operator ideals ${\cal I}_1$ and ${\cal I}_2$, we say that ${\cal I}_1$-AP = ${\cal I}_2$-AP if the Banach spaces having ${\cal I}_1$-AP are exactly the ones having ${\cal I}_2$-AP. The equivalence between (a) and (b) in Proposition \ref{propos} says that ${\cal I}$-AP = $\overline{\cal I}$-AP for every operator ideal $\cal I$. In particular,

{\begin{corollary} Let ${\cal I}_1$ and ${\cal I}_2$ be operator ideals. If $\overline{{\cal I}_1} = \overline{{\cal I}_2}$, then \linebreak ${\cal I}_1$-AP = ${\cal I}_2$-AP.
\end{corollary}

\begin{example}\label{example} \rm Since ${\cal F} \subseteq {\cal N}_p \subseteq \overline{{\cal F}} = {\cal A}$ \cite[Proposition 19.7.3]{jarchow}, it holds that ${\cal N}_p$-AP = AP whereas ${\cal F} \neq {\cal N}_p \neq \overline{{\cal F}} = {\cal A}$.
\end{example}



Let us see a few more interesting conditions that are equivalent to the $\cal I$-AP:





\begin{proposition}\label{charac}
Let $\mathcal{I}$ be an operator ideal. The following statements are equivalent for a Banach space $E$:\\
{\rm (a)} $E$ has the $\mathcal{I}$-approximation property.\\
{\rm (b)} For every Banach $F$, $\mathcal{L}(E;F)=\overline{\mathcal{I}(E;F)}^{\tau_c}$.\\
{\rm (c)} For every Banach $F$, $\mathcal{L}(F;E)=\overline{\mathcal{I}(F;E)}^{\tau_c}$.\\
{\rm (d)} $\sum_{n=1}^{\infty}x'_n(T(x_n))=0$ for every $T \in {\cal L}(E;E)$ whenever the sequences $(x_n)\subseteq E$ and
$(x'_n)\subseteq E'$ are such that
$\sum_{n=1}^{\infty}\|x'_n\|\|x_n\|<\infty$ and
 $\sum_{n=1}^{\infty}x'_n(T(x_n))=0$ for every $T\in \mathcal{I}(E;E)$. \\
{\rm (e)} $\sum_{n=1}^{\infty}x'_n(x_n)=0$ whenever the sequences $(x_n)\subseteq E$ and
$(x'_n)\subseteq E'$ are such that
$\sum_{n=1}^{\infty}\|x'_n\|\|x_n\|<\infty$ and
 $\sum_{n=1}^{\infty}x'_n(T(x_n))=0$ for every $T\in \mathcal{I}(E;E)$.
\end{proposition}

\begin{proof} (a) $\Longrightarrow$ (b) and (a) $\Longrightarrow$ (c) are straightforward. (b) $\Longrightarrow$ (a), (c) $\Longrightarrow $ (a) and (d) $\Longrightarrow $ (e) are obvious.

\medskip

\noindent (e) $\Longrightarrow$ (a) Let $\varphi \in ({\cal L}(E;E),{\tau_c})'$ be such that $\varphi(T) = 0$ for every $T \in {\cal I}(E;E)$. By Grothendieck's description \cite{Grothendieck} of the functionals belonging to $({\cal L}(E;E),{\tau_c})'$ (proofs can be found in \cite[Proposition 1.e.3]{lt} and \cite[Lemma VIII.3.3]{du}), there are sequences $(x_n)\subseteq E$ and
$(x'_n)\subseteq E'$ such that $\sum_{n=1}^{\infty}\|x'_n\|\|x_n\|<\infty$ and
 $\varphi(T) = \sum_{n=1}^{\infty}x'_n(T(x_n))$ for every $T\in \mathcal{L}(E;E)$. By (d) we have that $\varphi(id_E) = \sum_{n=1}^{\infty}x'_n(x_n)=0$. Hence $\varphi(id_E) =0$ for every functional $\varphi \in ({\cal L}(E;E),{\tau_c})'$ that vanishes on ${\cal I}(E;E)$. By the Hahn-Banach theorem (see, e.g., \cite[Corollary 2.2.20]{Megginson}) it follows that $id_E \in \overline{{\cal I}(E;E)}^{\tau_c}$.

\medskip

\noindent (a) $\Longrightarrow$ (d) Assume that (d) does not hold. In this case there are sequences $(x_n)\subseteq E$ and
$(x'_n)\subseteq E'$ such that
$\sum_{n=1}^{\infty}\|x'_n\|\|x_n\|<\infty$,
 $\sum_{n=1}^{\infty}x'_n(T(x_n))=0$ for every $T\in \mathcal{I}(E;E)$ and $\sum_{n=1}^{\infty}x'_n(U(x_n))\neq0$ for some $U\in \mathcal{L}(E;E)$. Defining
$$\varphi \colon \mathcal{L}(E;E) \longrightarrow \mathbb{K}~,~ \varphi(T)=\sum_{n=1}^{\infty}x'_n(T(x_n)),$$
by the above mentioned Grothendieck's description we know that $\varphi\in
(\mathcal{L}(E;E), \tau_c)'$. Thus $\varphi$ vanishes on $\mathcal{I}(E;E)$ but $\varphi(U) \neq 0$. Calling on Hahn-Banach once more we conclude that $U \notin \overline{\mathcal{I}(E;E)}^{\tau_c}$, which contradicts (a).
\end{proof}

As expected, $\cal I$-AP is inherited by complemented subspaces and is stable under the formation of finite cartesian products:

\begin{proposition} \label{complemented}
Let $\mathcal{I}$ be an operator ideal and $E$ be a Banach
space with the $\mathcal{I}$-approximation property. Then every complemented subspace of $E$ has the
$\mathcal{I}$-approximation property as well.
\end{proposition}

\begin{proposition}\label{sum} Let $\mathcal{I}$ be an operator ideal, $k \in \mathbb{N}$ and $E_1, \ldots, E_k$ be Banach spaces. Then the finite direct sum (or cartesian product) $E=\bigoplus_{n=1}^{k} E_{n}$ has the
$\mathcal{I}$-approximation property if and only if $E_1, \ldots, E_n$ have the $\mathcal{I}$-approximation property.
\end{proposition}

\section{Duality}
In this section we study the dual properties of the ${\cal I}$-approximation property. Given an operator ideal $\cal I$ and Banach spaces $E$ and $F$, define
 $$\mathcal{I}^{dual}(E;F)=\{S\in \mathcal{L}(E;F) \,\,{\rm such \, that~the~adjoint~operator}\,\, S'\in \mathcal{I}(F';E')\}.$$
It is well known that $\mathcal{I}^{dual}$ is an operator ideal. By $J_E$ we mean the canonical embedding from $E$ to $E''$.

\begin{proposition}\label{1propdualIAP} Let ${\mathcal{I}}_1$ and ${\mathcal{I}}_2$ be
operator ideals. If $E'$ has ${\mathcal{I}}_2$-AP, $F$ is reflexive and ${\mathcal{I}}_2(F';E')\subseteq
{\mathcal{I}}_1^{dual}(F';E')$, then $\mathcal{L}(E;F)=\overline{\mathcal{I}_1(E;F)}^{\tau_c}$.
\end{proposition}

\begin{proof}
Let $V\in
\mathcal{L}(E;F)$ and let $\varphi \in (\mathcal{L}(E;F), \tau_c)'$ be
such that $\varphi(T)=0$ for every $T\in \mathcal{I}_1(E;F)$. It is enough to show that $\varphi(V)=0$, because in this case $V\in \overline{\mathcal{I}_1(E;F)}^{\tau_c}$ by \cite[Corollary
2.2.20]{Megginson}. Calling on Grothendieck's description of $(\mathcal{L}(E;F), \tau_c)'$ once more, there are
sequences $(x_n)\subseteq E$ and $(y'_n)\subseteq F'$ such that
$\sum_{n=1}^{\infty}\|y'_n\|\|x_n\|<\infty$ and
$\varphi(U)=\sum_{n=1}^{\infty}y'_n(U(x_n))$. Let $S\in \mathcal{I}_2(F';E')$. By assumption we have that $S'\in
\mathcal{I}_1(E'';F'')$. From the reflexivity of $F$ we may define
$R:=(J_F)^{-1}\circ S'\circ J_E \in \mathcal{I}_1(E;F)$. For every
$u\in F''$ and $v\in F'$,
$$\langle u, v \rangle= \langle J_F((J_F)^{-1}(u)),v \rangle= \langle v, (J_F)^{-1}(u) \rangle.
$$
Observe that $\phi(\cdot)= \sum_n J_E(x_n)(\cdot)y'_n \in
(\mathcal{L}(F';E'), \tau_c)'$ and
\begin{eqnarray*}
\phi(S)&=& \sum_{n=1}^{\infty}
J_E(x_n)(S)y'_n=\sum_{n=1}^{\infty}J_E(x_n)(S(y'_n))
=\sum_{n=1}^{\infty}\langle J_E(x_n),S(y'_n) \rangle\\
&=&\sum_{n=1}^{\infty} \langle S'\circ J_E(x_n),y'_n \rangle =
\sum_{n=1}^{\infty} \langle y'_n, (J_F)^{-1}\circ S'\circ J_E(x_n) \rangle\\
&=&\sum_{n=1}^{\infty} \langle y'_n,R(x_n) \rangle=
\sum_{n=1}^{\infty}y'_n(R(x_n))=0.
\end{eqnarray*}
So $\phi(S) = 0$ for every $S\in \mathcal{I}_2(F';E').$ Since $E'$ has
${\cal{I}}_2$-AP, $\mathcal{L}(F';E')=\overline{\mathcal{I}_2(F';E')}^{\tau_c}$ by Proposition \ref{charac}.
Therefore \cite[Corollary 2.2.20]{Megginson} yields $\phi(V')=0$. Thus \begin{eqnarray*}0&=& \phi(V')
=\sum_{n=1}^{\infty}J_E(x_n)(V'(y'_n))
=\sum_{n=1}^{\infty}\langle J_E(x_n),V'(y'_n) \rangle\\
&=&\sum_{n=1}^{\infty} \langle V'(y'_n),x_n\rangle =\sum_{n=1}^{\infty} \langle y'_n,V(x_n) \rangle= \varphi(V).
\end{eqnarray*}
The proof is complete.
\end{proof}

\begin{theorem} \label{propdualIAP} Let ${\mathcal{I}}_1$ and ${\mathcal{I}}_2$ be operator ideals such that either
 ${\mathcal{I}}_2\subseteq {\mathcal{I}}_1^{dual}$ or ${\mathcal{I}}_2^{dual}\subseteq {\mathcal{I}}_1$ and $E$ be a reflexive Banach
space.\\
{\rm (a)} If $E'$ has ${\mathcal{I}}_2$-AP then $E$ has ${\mathcal{I}}_1$-AP.\\
{\rm (b)} If $E$ has ${\mathcal{I}}_2$-AP then $E'$ has
${\mathcal{I}}_1$-AP.
\end{theorem}

\begin{proof}Assume that ${\mathcal{I}}_2^{dual}\subseteq {\mathcal{I}}_1$.
Let $u \in {\cal I}_2(E';E')$. Since $E$ is reflexive, $((J_E)^{-1}
\circ u' \circ J_E)' = u \in {\cal I}_2(E';E')$, hence $(J_E)^{-1}
\circ u' \circ J_E \in \mathcal{I}_2^{dual}(E;E)$. By assumption we
have $(J_E)^{-1} \circ u' \circ J_E \in \mathcal{I}_1(E;E)$. Then
$u' = J \circ (J_E)^{-1} \circ u' \circ J_E \circ (J_E)^{-1} \in
\mathcal{I}_1(E;E)$ by the ideal property, that is, $ u \in
\mathcal{I}_1^{dual}(E';E')$. We have just proved that
$\mathcal{I}_2(E';E')\subseteq \mathcal{I}_1^{dual}(E';E')$. Since
this condition holds trivially if $\mathcal{I}_2\subseteq
\mathcal{I}_1^{dual}$, in both cases we have
$\mathcal{I}_2(E';E')\subseteq
\mathcal{I}_1^{dual}(E';E')$.\\
\indent (a) Suppose that $E'$ has ${\cal{I}}_2$-AP. Since $E$ is
reflexive, we have $\mathcal{L}(E;E)=\overline{\mathcal{I}_1(E;E)}^{\tau_c}$ from Proposition \ref{1propdualIAP}, so $E$
has ${\cal{I}}_1$-AP.

(b) Suppose that $E$ has ${\cal{I}}_2$-AP. Since $E$
and $E''$ are isometrically isomorphic, it follows that $E''$ has
${\cal{I}}_2$-AP. Hence $E'$ has ${\cal{I}}_1$-AP by (a).
\end{proof}

\begin{corollary} \label{dualIAP}
Let $\mathcal{I}$ be an operator ideal such that
either $\mathcal{I}\subseteq \mathcal{I}^{dual}$ or
$\mathcal{I}^{dual}\subseteq \mathcal{I}$ and $E$ be a reflexive
Banach space. Then $E'$ has the $\mathcal{I}$-approximation property
if and only if $E$ has the $\mathcal{I}$-approximation property.
\end{corollary}

Given $1 \leq p < \infty$, $p^*$ stands for the conjugate of $p$, that is $\frac1p + \frac{1}{p^*}=1$. For the definition of the adjoint ideal ${\cal I}^*$ of the operator ideal $\cal I$, see, e.g., \cite[p. 132]{DJT}.

\begin{example}\rm Let us see that there is plenty of ideals satisfying the conditions of Theorem \ref{propdualIAP} and Corollary \ref{dualIAP}.\\
(i) ${\cal N}_1^{dual} \subseteq {\cal J}$ \cite[Ex. 16.9]{klaus}, ${\cal SS}^{dual} \subseteq {\cal SC}$ and ${\cal SC}^{dual} \subseteq {\cal SS}$ \cite[1.18]{DJP}, $\Gamma_p^{dual} = \Gamma_{p^*}$ \cite[p. 186]{DJT}, $\Pi_1^{dual} = \Gamma_1^*$ \cite[Corollary 9.5]{DJT}, ${\cal T}_p \subseteq {\cal C}_{p^*}^{dual}$ and ${\cal C}_{p^*} \circ {\cal KC} \subseteq {\cal T}_p^{dual}$ for $1 < p \leq 2$ \cite[31.2]{klaus}, ${\cal N}_1^{dual} \subseteq {\cal QN}$ \cite[Ex. 9.13(b)]{klaus}, $\Pi_{r,p,q}^{dual} = \Pi_{r,q,p}$ \cite[Theorem 17.1.5]{Pietsch}, ${\cal L}_{p,q}^{dual} = {\cal L}_{q,p}$ \cite[p. 68]{cdr}, ${\cal K}_p = {\cal QN}^{dual}_p$ \cite{delgadostudia}.\\
(ii) The following ideals are completely symmetric (that is $\mathcal{I}= \mathcal{I}^{dual}$): ${\cal F}, {\cal A}, {\cal K}, {\cal W}$ \cite[Proposition 4.4.7]{Pietsch}, $\cal J$ \cite[Corollary 10.2.2]{klaus}, $\cal SN$ \cite[Theorem 19.9.3]{jarchow}, ${\cal U}_p$, $0<p<\infty$ \cite[Theorem 14.2.5]{Pietsch} and $\cal KC$ \cite[31.1]{klaus}.\\
(iii) The following ideals satisfy $\mathcal{I}\subseteq \mathcal{I}^{dual}$: ${\cal N}_1$ \cite[9.9]{klaus} and $\cal D$ \cite[Proposition 4.4.10]{Pietsch}.\\
(iv) The following ideals satisfy $\mathcal{I}^{dual}\subseteq \mathcal{I}$: $\cal S$ \cite[Proposition 4.4.8]{Pietsch} and ${\cal DP}$ \cite[1.15]{DJP}.
\end{example}

Our next aim is to show that the implication $E'$ has $\cal{I}-$AP $\Longrightarrow$ $E$ has $\cal{I}$-AP holds in some situations not covered by Corollary \ref{dualIAP}. A couple of concepts defined in \cite{choi-kim-outro} are needed:

\begin{definition}\rm Let $E$ be a Banach space. Consider in $\mathcal{L}(E;E)$ the topology, called $\nu$, for which a net $(T_\alpha)$ in
$\mathcal{L}(E;E)$ converges to $T \in \mathcal{L}(E;E)$ if and only if $$\sum_{n=1}^{\infty}x'_n(T_{\alpha}(x_n)) \longrightarrow \sum_{n=1}^{\infty}x'_n(T(x_n))$$
for every $(x_n)\subseteq E$ and $(x^{'}_n)\subseteq E^{'}$ satisfying
$\sum_{n=1}^{\infty}\|x'_n\|\|x_n\|<\infty$. In this case we write $T_\alpha \stackrel{\nu}\longrightarrow T$. It is immediate that the $\tau_c$-topology is
stronger than the $\nu$-topology on $\mathcal{L}(E;E)$.\\
\indent The {\it $weak^{*}$-topology} on $\mathcal{L}(E';E')$ is the topology for which a net $(T_\alpha)$ in $\mathcal{L}(E';E')$ converges to $T \in \mathcal{L}(E';E')$ if and only if
$$\sum_{n=1}^{\infty}(T_{\alpha}(x'_n))x_n \longrightarrow \sum_{n=1}^{\infty}(T (x'_n))x_n$$
for every $(x_n)\subseteq E$ and $(x^{'}_n)\subseteq E^{'}$ satisfying
$\sum_{n=1}^{\infty}\|x'_n\|\|x_n\|<\infty$. In this case we write $T_\alpha \stackrel{weak^{*}} \longrightarrow  T$.\\
\indent The topology $\nu$ is stronger than the $weak^{*}$-topology on $\mathcal{L}(E';E')$. Moreover, for $T$ and a net
$(T_\alpha)$ in $\mathcal{L}(E;E)$,
\begin{equation}\label{equation}T_\alpha \stackrel{\nu} \longrightarrow  T \Longleftrightarrow T'_\alpha \stackrel{weak^{*}} \longrightarrow  T'.
\end{equation}

Given a Banach space $E$, be $w^*$ we mean the ordinary weak* topology on $E'$. For a given operator ideal $\cal I$, by ${\cal I}_{w^*}(E';E')$ we denote the set of all operators belonging to ${\cal I}(E';E')$ which are $w^*$-to-$w^*$ continuous. 
The dual space $E'$ is said to have the {\it weak* density} for $\cal{I}$ (in short, $E$ has $\cal{I}$-W*D) if
$$\mathcal{I}(E';E')\subseteq
\overline{\mathcal{I}_{w^*}(E';E')}^{weak^*}.$$
\end{definition}

\begin{example}\rm There are nonreflexive dual Banach spaces having $\cal{I}$-W*D for every operator ideal $\cal I$. In \cite[Proposition 2.7(a)]{choi-kim-outro} it is proved that $\ell_1$ has $\cal{K}$-W*D. The only feature of compact operators used in the proof is the ideal property, so the same lines prove that $\ell_1 = (c_0)'$ is a nonreflexive dual Banach space having $\cal{I}$-W*D for every operator ideal $\cal I$.
\end{example}

So, formally Corollary \ref{dualIAP} does not apply to dual spaces having $\cal{I}$-W*D. In this direction we have:




\begin{proposition} Let $E$ be a Banach space and let $\mathcal{I}$ be an operator ideal such that $ \mathcal{I}^{dual} \subseteq \mathcal{I}$.
If $E'$ has $\cal{I}$-AP and $\cal{I}$-W*D, then $E$ has $\cal{I}$-AP.
\end{proposition}
\begin{proof}
Let  $(x_n)\subseteq E$ and $(x^{'}_n)\subseteq E^{'}$ be sequences
such that $\sum_{n=1}^{\infty}\|x'_n\|\|x_n\|<\infty$ and
$\sum_{n=1}^{\infty}x'_n(T(x_n))=0$ for every $T\in
\mathcal{I}(E;E)$. We know that $id_{E'}\in \overline{\mathcal{I}(E';E')}^{\tau_c}$ because $E'$ has $\cal{I}$-AP, and that $\mathcal{I}(E';E')\subseteq
\overline{\mathcal{I}_{w^*}(E';E')}^{weak^*}$ because $E$ has $\cal{I}-W^*D$. Thus $id_{E'}\in \overline{\mathcal{I}_{w^*}(E';E')}^{weak^*}$ and then
there is a net $(S_{\alpha})\subseteq \mathcal{I}_{w^*}(E';E')$
such that $S_{\alpha} \stackrel{weak^*}\longrightarrow id_{E'}$. For each $\alpha$, since $S_\alpha$ is $w^*$-to-$w^*$ continuous, there is $T_\alpha \in {\cal L}(E;E)$ such that $T_\alpha ' = S_\alpha$ (see \cite[Ex.\,\,3.20]{varios}). We know that $S_\alpha \in \mathcal{I}(E';E')$, so the condition $ \mathcal{I}^{dual} \subseteq \mathcal{I}$ yields that $T_\alpha  \in {\cal I}(E;E)$ for every $\alpha$.
 From (\ref{equation}) and
$$T_{\alpha}' = S_\alpha
\stackrel{weak^*}\longrightarrow id_{E'} = (id_E)'$$
we get $T_{\alpha}
\stackrel{\nu}\longrightarrow id_{E}$. So, by the definition of the $\nu$-convergence,
$$\sum_{n=1}^{\infty}x'_n(T_{\alpha}(x_n)) \longrightarrow
\sum_{n=1}^{\infty}x'_n(id_E(x_n)) = \sum_{n=1}^{\infty}x'_n(x_n) .$$
But
$\sum_{n=1}^{\infty}x'_n(T_{\alpha}(x_n))=0$ for every $\alpha$ because each $T_\alpha  \in {\cal I}(E;E)$, therefore
$\sum_{n=1}^{\infty}x'_n(x_n)=0$. By Proposition \ref{charac} it follows that $E$
has $\cal{I}$-AP.
\end{proof}

}
\section{Tensor stability}

In this section we study the stability of $\cal I$-AP under the formation of projective tensor products. By $E_1 \hat\otimes_\pi \cdots \hat\otimes_\pi E_n$ we mean the completed projective tensor product of $E_1, \ldots, E_n$ ($\hat\otimes_\pi^n E$ if $E = E_1 = \cdots = E_n$), and by $\hat\otimes_\pi^{n,s}E$ the completed $n$-fold symmetric projective tensor product of $E$.
\begin{definition}\rm Given continuous linear operators $u_j \colon E_j \longrightarrow F_j$, $j = 1, \ldots, n$, by $u_1 \otimes \cdots\otimes u_n$ we denote the (unique) continuous linear operator from $E_1 \hat\otimes_\pi \cdots \hat\otimes_\pi E_n$ to $F_1 \hat\otimes_\pi \cdots \hat\otimes_\pi F_n$ such that
$$u_1 \otimes \cdots\otimes u_n(x_1 \otimes \cdots \otimes x_n) = u_1(x_1) \otimes \cdots \otimes u_n(x_n) $$
for every $x_1 \in E_1, \ldots, x_n \in E_n$.
\end{definition}
The proof of the stability of the approximation property with respect to the formation of projective tensor products relies heavily on the fact that $u_1 \otimes \cdots\otimes u_n$ is a finite rank operator whenever $u_1, \ldots, u_n$ are finite rank operators. Let us see that this does not hold for arbitrary operator ideals:
\begin{example}\rm The identity operator $id_{\ell_2}$ is weakly compact but $id_{\ell_2} \otimes id_{\ell_2} = id_{\ell_2 \hat\otimes_\pi \ell_2}$ is not weakly compact because $\ell_2 \hat\otimes_\pi \ell_2$ fails to be reflexive.
\end{example}
In order to settle this difficulty we need the following methods of generating ideals of multilinear mappings from operator ideals. By ${\cal L}(E_1, \ldots, E_n;F)$ we denote the space of continous $n$-linear mappings from $E_1 \times \cdots \times E_n$ to $F$ endowed with the usual sup norm.
\begin{definition}\rm Let $\mathcal{I}, \mathcal{I}_1, \ldots, \mathcal{I}_n $ be operator ideals.\\
(a) (Factorization Method) A continuous $n$-linear mapping $A\in \mathcal{L}(E_1, \ldots, E_n;F)$
  is said to be of type  $\mathcal{L}[\mathcal{I}_1, \ldots, \mathcal{I}_n]$ if
  there are Banach spaces $G_1, \ldots, G_n$, linear operators  $u_j\in
\mathcal{I}_j(E_j;G_j)$, $j = 1, \ldots, n$, and an $n$-linear mapping $B\in \mathcal{L}(G_1, \ldots, G_n;F)$ such that $A=B\circ (u_1, \ldots, u_n)$. In this case we write $A \in {\mathcal{L}[\mathcal{I}_1, \ldots, \mathcal{I}_n]}(E_1, \ldots, E_n;F)$. If $\mathcal{I} = \mathcal{I}_1 = \cdots = \mathcal{I}_n $ we simply write $\mathcal{L}[\mathcal{I}]$.\\
(b) (Composition ideals) A continuous
$n$-linear mapping $A\in \mathcal{L}(E_1, \ldots, E_n;F)$ belongs to
$\mathcal{I}\circ \mathcal{P}$ if there are a Banach space
$G$, an $n$-linear mapping $B\in \mathcal{L}(E_1, \ldots, E_n;G)$ and a linear operator
$u\in \mathcal{I}(G;F)$ such that $A =u\circ B$. In this case we write $A\in
\mathcal{I}\circ\mathcal{L}(E_1, \ldots, E_n;F)$.\\
\indent For details and examples we refer to \cite{Botelho1,Botelho0}.
\end{definition}
\begin{proposition}\label{prodtens} Let $\mathcal{I}, \mathcal{I}_1, \ldots, \mathcal{I}_n $ be operator ideals such that $\mathcal{L}[\mathcal{I}_1, \ldots, \mathcal{I}_n] \subseteq \mathcal{I}\circ\mathcal{L}$. If $E_1$ has $\mathcal{I}_1$-AP, $\ldots$, $E_n$ has $\mathcal{I}_n$-AP, then $E_1 \hat\otimes_\pi \cdots \hat\otimes_\pi E_n$ has $\cal I$-AP.
\end{proposition}
\begin{proof} Let $K$ be a compact subset of $E_1 \hat\otimes_\pi \cdots \hat\otimes_\pi E_n$. By \cite[Corollary 3.5.1]{klaus} there are compact sets $K_1 \subseteq E_1, \ldots, K_n \subseteq E_n$ such that $K$ is contained in the closure of the absolutely convex hull of $K_1 \otimes \cdots \otimes K_n :=$ $ \{x_1 \otimes \cdots \otimes x_n : x_1 \in K_1, \ldots, x_n \in K_n\}$. Since compact sets are bounded there is $M > 0$ such that $\|x_j \| \leq M$ for every $x_j \in E_j$, $j = 1, \ldots, n$. Let $\varepsilon > 0$. As $E_1$ has $\mathcal{I}_1$-AP, there is an operator $u_1 \in {\cal I}_1(E_1;E_1)$ such that $\|u_1(x_1) - x_1\| < \frac{\varepsilon}{2nM^{n-1}}$ for every $x_1 \in K_1$. As $E_2$ has $\mathcal{I}_2$-AP, there is an operator $u_2 \in {\cal I}_2(E_2;E_2)$ such that $\|u_2(x_2) - x_2\| < \frac{\varepsilon}{2nM^{n-1}\|u_1\|}$ for every $x_2 \in K_2$. Repeating the procedure we obtain operators $u_j \in {\cal I}_j(E_j;E_j)$ such that
$$\|u_j(x_j) - x_j\| < \frac{\varepsilon}{2nM^{n-1}\|u_1\|\cdots \|u_{j-1}\|}$$
for every $x_j \in K_j$, $j = 1, \ldots, n$. Consider the canonical $n$-linear mapping $\sigma_n \colon E_1 \otimes \cdots \times E_n \longrightarrow  E_1 \hat\otimes_\pi \cdots \hat\otimes_\pi E_n$ given by $\sigma_n(x_1 , \ldots ,x_n) = x_1 \otimes \cdots \otimes x_n$ and observe that $\sigma_n \circ (u_1, \ldots, u_n) \in {\mathcal{L}[\mathcal{I}_1, \ldots, \mathcal{I}_n]}(E_1, \ldots, E_n;E_1 \hat\otimes_\pi \cdots \hat\otimes_\pi E_n)$. By assumption we have $\sigma_n \circ (u_1, \ldots, u_n) \in \mathcal{I} \circ {\mathcal{L}}(E_1, \ldots, E_n;E_1 \hat\otimes_\pi \cdots \hat\otimes_\pi E_n)$. Calling $T$ the linearization of $\sigma_n \circ (u_1, \ldots, u_n)$, by \cite[Proposition 3.2(a)]{Botelho0} we have that $T \in {\cal I}(E_1 \hat\otimes_\pi \cdots \hat\otimes_\pi E_n;E_1 \hat\otimes_\pi \cdots \hat\otimes_\pi E_n)$. For every $x_1 \in E_1, \ldots, x_n \in E_n$,
\begin{eqnarray*}T(x_1 \otimes \cdots \otimes x_n) &=& \sigma_n \circ (u_1, \ldots, u_n)(x_1, \ldots, x_n)\\
&=& \sigma_n (u_1(x_1), \ldots, u_n(x_n))\\
&=& u_1(x_1) \otimes \cdots \otimes u_n(x_n)\\
&=&  u_1 \otimes \cdots \otimes u_n(x_1 \otimes \cdots \otimes x_n).
\end{eqnarray*}
As both $T$ and $u_1 \otimes \cdots \otimes u_n $ are linear it follows that $T = u_1 \otimes \cdots \otimes u_n $, hence $u_1 \otimes \cdots \otimes u_n \in {\cal I}(E_1 \hat\otimes_\pi \cdots \hat\otimes_\pi E_n;E_1 \hat\otimes_\pi \cdots \hat\otimes_\pi E_n)$. We shall denote the projective norm of a tensor $z \in E_1 \hat\otimes_\pi \cdots \hat\otimes_\pi E_n$ by $\|z\|$ instead of $\pi(z)$. Given $x_1 \in K_1, \ldots, x_n \in K_n$,
$$\|u_1 \otimes \cdots \otimes u_n (x_1 \otimes \cdots \otimes x_n) - x_1 \otimes \cdots\otimes x_n \|= \|u_1(x_1) \otimes \cdots \otimes u_n (x_n) - x_1 \otimes \cdots\otimes x_n \|$$
$$= \|u_1(x_1) \otimes \cdots \otimes u_n (x_n) - \sum_{j=1}^{n-1} u_1(x_1) \otimes \cdots \otimes u_j(x_j) \otimes x_{j+1} \otimes \cdots \otimes x_n \hspace*{30em}$$
$$ + \sum_{j=1}^{n-1} u_1(x_1) \otimes \cdots \otimes u_j(x_j) \otimes x_{j+1} \otimes \cdots \otimes x_n   - x_1 \otimes \cdots\otimes x_n \| \hspace*{5em} $$
$$= \left\|\sum_{j=1}^n u_1(x_1)\otimes \cdots \otimes u_{j-1}(x_{j-1}) \otimes (u_j(x_j) - x_j)\otimes x_{j+1} \otimes \cdots \otimes x_n   \right\|\hspace*{30em}$$
$$\leq \sum_{j=1}^n \|u_1(x_1)\otimes \cdots \otimes u_{j-1}(x_{j-1}) \otimes (u_j(x_j) - x_j)\otimes x_{j+1} \otimes \cdots \otimes x_n \| \hspace*{30em}$$
$$ \leq \sum_{j=1}^n \|u_1\|\|x_1\| \cdots \|u_{j-1}\|\|x_{j-1}\|\|u_j(x_j) - x_j\|\|x_{j+1}\|  \cdots \|x_n \| \hspace*{30em}$$
$$< \sum_{j=1}^n \|u_1\| \cdots \|u_{j-1}\|M^{n-1}\frac{\varepsilon}{2nM^{n-1}\|u_1\|\cdots \|u_{j-1}\|} = \frac{\varepsilon}{2}. \hspace*{30em}$$
In summary,
$$\|u_1 \otimes \cdots \otimes u_n (x_1 \otimes \cdots \otimes x_n) - x_1 \otimes \cdots\otimes x_n \| < \frac{\varepsilon}{2}, $$
for every $x_1 \in K_1, \ldots, x_n \in K_n$. Take $z$ in the absolutely convex hull of $K_1 \otimes \cdots \otimes K_n$. Then $z = \sum_{j=1}^k \lambda_j x_j^1 \otimes \cdots \otimes x_j^n$, where $k \in \mathbb{N}$, $\lambda_1,\ldots, \lambda_k$ are scalars with $|\lambda_1| + \cdots + |\lambda_k |\leq 1$, and $x_j^m \in K_m$ for $j = 1, \ldots, k, m = 1, \ldots, n$. Then
$$\|u_1 \,\otimes \cdots \otimes u_n(z)~ - ~z\|  \hspace*{30em}$$
$$ = \left\|u_1 \otimes \cdots \otimes u_n\left(\sum_{j=1}^k \lambda_j x_j^1 \otimes \cdots \otimes x_j^n \right)- \sum_{j=1}^k \lambda_j x_j^1 \otimes \cdots \otimes x_j^n  \right\| \hspace*{30em}$$
$$= \left\| \sum_{j=1}^k \lambda_j \left(u_1 \otimes \cdots \otimes u_n\left( x_j^1 \otimes \cdots \otimes x_j^n \right)- x_j^1 \otimes \cdots \otimes x_j^n \right)  \right\| \hspace*{30em}$$
 $$\leq   \sum_{j=1}^k |\lambda_j|\left\|\left(u_1 \otimes \cdots \otimes u_n\left( x_j^1 \otimes \cdots \otimes x_j^n \right)- x_j^1 \otimes \cdots \otimes x_j^n \right)  \right\|   \hspace*{30em}$$
$$<  \frac{\varepsilon}{2}\sum_{j=1}^k |\lambda_j| = \frac{\varepsilon}{2}.\hspace*{30em}$$
By continuity we have that
$$\|u_1 \otimes \cdots \otimes u_n(z) - z\|\leq \frac{\varepsilon}{2} < \varepsilon$$
for every $z$ in the closure of the absolutely convex hull of $K_1 \otimes \cdots \otimes K_n$, hence for every $z \in K$.
\end{proof}

As to ideals satisfying the conditions above we have:
\begin{example}\rm (a) It is plain that $\mathcal{L}[\mathcal{F}] \subseteq \mathcal{F}\circ\mathcal{L}$ and $\mathcal{L}[\mathcal{S}] \subseteq \mathcal{S}\circ\mathcal{L}$. \\
(b) $\mathcal{L}[\mathcal{N}_1] \subseteq \mathcal{N}_1\circ\mathcal{L}$ \cite[Theorem 3.7]{holub1} (see also \cite[Proposition 17.3.9]{jarchow}).\\
(c) $\mathcal{L}[\mathcal{J}] \subseteq \mathcal{J}\circ\mathcal{L}$ \cite[Theorem 2]{holub}.\\
(d) Let ${\cal L}_{\cal K}$ denote the ideal of compact multilinear mappings (bounded sets are sent to relatively compact sets). Pe{\l}czy\'nski \cite{pel} proved that ${\cal K} \circ {\cal L} = {\cal L}_{\cal K}$. Now it follows easily that $\mathcal{L}[\mathcal{K}] \subseteq \mathcal{K}\circ\mathcal{L}$. So the projective tensor product of spaces with CAP has CAP too.\\
(e) $\mathcal{L}[\mathcal{L}_{\infty, q, \gamma}] \subseteq \mathcal{L}_{\infty, q, \gamma}\circ\mathcal{L}$ for $0<q\leq1$ and $-1/q < \gamma < \infty$ \cite[Theorem 3.1]{cobos}.\\
(f) $\mathcal{L}[\mathcal{L}_{1,q}] \subseteq \mathcal{L}_{1, q}\circ\mathcal{L}$ for $q> 1$ and $\mathcal{L}[\mathcal{K}_{1,p}] \subseteq \mathcal{K}_{1, p}\circ\mathcal{L}$ for $p \geq 1$ \cite[Theorem 2.1]{cdr}.\\
(g) It is unknown if the projective tensor product of Schur spaces is a Schur space (see, e.g., \cite{pams}), so it is unknown if  $\mathcal{L}[\mathcal{CC}] \subseteq \mathcal{CC}\circ\mathcal{L}$.
\end{example}

Here are other concrete situations to which Proposition \ref{prodtens} applies:
\begin{lemma}\label{racher} Let $n \in \mathbb{N}$.\\
 {\rm (a)} If $1 \leq p_1, \ldots, p_n < \infty$, then ${\cal L}[{\cal W}, {\cal I}_1, \ldots, {\cal I}_n] \subseteq {\cal W} \circ {\cal L}$ where ${\cal I}_j$ is either ${\cal K}$ or $\Pi_{p_j}$, $j = 1, \ldots, n$.\\
 {\rm (b)} ${\cal L}[\Pi_1, {\cal J}, \stackrel{(n)}{\ldots}, {\cal J}] \subseteq \Pi_1 \circ {\cal L}$.\\
{\rm (c)} ${\cal L}[{\cal QN}, {\cal N}_1, \stackrel{(n)}{\ldots}, {\cal N}_1] \subseteq {\cal QN} \circ {\cal L}$.\\
{\rm (d)} If $p_1 > p_j$ for $j = 2, \ldots, n$, then ${\cal L}[{\cal U}_{p_1}, {\cal U}_{p_2}, \ldots, {\cal U}_{p_n}] \subseteq {\cal U}_{p_1} \circ {\cal L}$.
\end{lemma}
\begin{proof} (a) Given an $n$-linear mapping $A \in {\cal L}[{\cal W}, {\cal I}_1, \ldots, {\cal I}_n](E,E_1, \ldots, E_;F)$, write $A = B \circ (u,u_1, \ldots, u_n)$ with $u\in
\mathcal{W}(E;G)$, $u_j\in
\mathcal{I}_j(E_j;G_j)$, $j = 1, \ldots, n$, and $B\in \mathcal{L}(G, G_1, \ldots, G_n;F)$. Since $u$ is weakly compact and $u_1$ is either compact or absolutely $p_1$-summing, by a result of Racher \cite{racher} we have that $u \otimes u_1$ is weakly compact. As $u_2$ is either compact or absolutely $p_2$-summing and the projective tensor norm is associative, $u \otimes u_1 \otimes u_2 = (u \otimes u_1) \otimes u_2$ is weakly compact by the same result of \cite{racher}. Repeating this procedure finitely many times we conclude that $u \otimes u_1 \otimes \cdots \otimes u_n$ is weakly compact. Denoting by $\sigma_{n+1}\colon E \times E_1 \times \cdots \times E_n \longrightarrow E \hat\otimes_\pi E_1 \hat\otimes_\pi \cdots \hat\otimes_\pi E_n  $ the canonical $(n+1)$-linear mapping and by $T$ the linearization of $B$, we conclude that $A = T \circ (u \otimes u_1 \otimes \cdots \otimes u_n) \circ \sigma_{n+1}$. It follows that $A \in {\cal W} \circ {\cal L}(E,E_1, \ldots, E_;F)$ because $T \circ (u \otimes u_1 \otimes \cdots \otimes u_n)$ is weakly compact by the ideal property.

\medskip

\noindent For (b), (c) and (d), repeat the proof above using, instead of Racher's result, the following results: in (b) and (c), two results due to Holub \cite{holub}: (i) If $u_1 \in \Pi_1(E_1;F_1)$ and $u_2 \in {\cal J}(E_2;F_2)$ then $u_1 \otimes u_2 \in \Pi_1(E_1 \hat\otimes_\pi E_2; F_1 \hat\otimes_\pi F_2)$; (ii) If $u_1 \in {\cal QN}(E_1;F_1)$ and $u_2 \in {\cal N}_1(E_2;F_2)$ then $u_1 \otimes u_2 \in {\cal QN}(E_1 \hat\otimes_\pi E_2; F_1 \hat\otimes_\pi F_2)$. In (d), the following result, which appears in K\"onig \cite[p. 79]{konig} and is credited to Pietsch \cite{pietscharchiv}: if $u_1 \in {\cal U}_{p_1}(E_1;F_1)$, $u_2 \in {\cal U}_{p_2}(E_2;F_2)$ and $p_1 > p_2$, then $u_1 \otimes u_2 \in {\cal U}_{p_1}(E_1 \hat\otimes_\pi E_2; F_1 \hat\otimes_\pi F_2)$.
\end{proof}
Combining Proposition \ref{prodtens} and Lemma \ref{racher} we get:
\begin{proposition}~\\
 {\rm (a)} Let $E_1, \ldots, E_n$ be Banach spaces, one of which with WCAP and the others $E_j$ with either CAP or $\Pi_{p_j}$-AP for some $1 \leq p_j < \infty$. Then $E_1 \hat\otimes_\pi \cdots \hat\otimes_\pi E_n$ has WCAP.\\
 {\rm (b)} Let $E_1, \ldots, E_n$ be Banach spaces, one of which with $\Pi_1$-AP and the others with $\cal J$-AP. Then $E_1 \hat\otimes_\pi \cdots \hat\otimes_\pi E_n$ has $\Pi_1$-AP.\\
{\rm (c)} Let $E_1, \ldots, E_n$ be Banach spaces, one of which with $\cal QN$-AP and the others with AP. Then $E_1 \hat\otimes_\pi \cdots \hat\otimes_\pi E_n$ has AP.\\
{\rm (d)} Let $0 < p_1, \ldots, p_n$. If $E_1, \ldots, E_n$ are Banach spaces, each $E_j$ with ${\cal U}_{p_j}$-AP, then $E_1 \hat\otimes_\pi \cdots \hat\otimes_\pi E_n$ has ${\cal U}_{p_k}$-AP if $p_k > p_j$ for every $j \neq k$.
\end{proposition}
\begin{example}\rm Let $E$ be a Banach space with WCAP but not with CAP (see the Introduction), and $E_1, \ldots, E_n$ be Banach spaces such that each $E_j$ has either CAP or $\Pi_{p_j}$-AP, $1 \leq p_1, \ldots, p_n < \infty$. Then $E  \hat\otimes_\pi E_1 \hat\otimes_\pi \cdots \hat\otimes_\pi E_n$ has WCAP.
\end{example}


\begin{corollary}\label{cortens} Let $\cal I$ be an operator ideal such that $\mathcal{L}[\mathcal{I}] \subseteq \mathcal{I}\circ\mathcal{L}$. The following are equivalent for a Banach space $E$:\\
{\rm (a)} $E$ has $\cal I$-AP.\\
\medskip
{\rm (b)} $\hat\otimes_\pi^n E$ has $\cal I$-AP for every $n \in \mathbb{N}$.\\
\medskip
{\rm (c)} $\hat\otimes_\pi^n E$ has $\cal I$-AP for some $n \in \mathbb{N}$.\\
\medskip
{\rm (d)} $\hat\otimes_\pi^{n,s}E$ has $\cal I$-AP for every $n \in \mathbb{N}$.\\
\medskip
{\rm (e)} $\hat\otimes_\pi^{n,s}E$ has $\cal I$-AP for some $n \in \mathbb{N}$.
\end{corollary}
\begin{proof} (a) $\Longrightarrow$ (b) follows from Proposition \ref{prodtens}; (b) $\Longrightarrow$ (c) is obvious; (c) $\Longrightarrow$ (a) follows from Proposition \ref{complemented} because $E$ is obviously a complemented subspace of $\hat\otimes_\pi^n E$; (b) $\Longrightarrow$ (d) follows from Proposition \ref{complemented} because $\hat\otimes_\pi^{n,s}E$ a complemented subspace of $\hat\otimes_\pi^n E$ via the symmetrization operator; (d) $\Longrightarrow$ (e) is obvious; (e) $\Longrightarrow$ (a) follows from Proposition \ref{complemented} because $E$ is a complemented subspace of $\hat\otimes_\pi^{n,s}E$ (see \cite[Corollary 4]{blasco}).
\end{proof}

\section{Polynomial ideals and the $\mathcal{I}$-AP}

The symbol $\mathcal{P}(^n E;F)$ stands for the space of continuous $n$-homogeneous polynomials from $E$ to $F$. A {\it polynomial ideal} is a subclass $\cal Q$ of the class of all
continuous homogeneous polynomials between Banach spaces such that, for every $n \in \mathbb{N}$ and Banach spaces $E$ and $F$, the component
${\cal Q}(^nE;F):=\mathcal{P}(^nE;F)\cap {\cal Q}$ satisfy:\\
(a) ${\cal Q}(^n E;F)$ is a linear subspace of $\mathcal{P}(^n E;F)$ which contains the $n$-homogeneous
polynomials of finite type,\\
(b) If $T\in \mathcal{L}(G;E)$, $P\in
{\cal Q}(^nE;F)$ and $S\in \mathcal{L}(F;H)$, then
$S\circ P\circ T \in {\cal Q}(^nG;H)$.



There are different ways to construct a polynomial ideal from a given operator ideal $\mathcal{I}$.
Let us see three of such methods (see \cite{Botelho1,Botelho0}):

\begin{definition}\rm Let $\mathcal{I}$ be an operator ideal.\\
(a) (Factorization Method) A continuous $n$-homogeneous polynomial  $P\in \mathcal{P}(^n E;F)$
  is said to be of type  $\mathcal{L}[\mathcal{I}]$ if
  there are a Banach space $G$, a linear operator  $u\in
\mathcal{I}(E;G)$ and a polynomial  $Q\in \mathcal{P}(^n G;F)$ such that $P=Q\circ u$. In this case we write $P \in {\cal P}_{\mathcal{L}[\mathcal{I}]}(^nE;F)$\\
(b) (Composition ideals) A continuous
$n$-homogeneous polynomial $P\in \mathcal{P}(^n E;F)$ belongs to
$\mathcal{I}\circ \mathcal{P}$ if there are a Banach space
$G$, a polynomial $Q\in \mathcal{P}(^n G;F)$ and a linear operator
$u\in \mathcal{I}(E;G)$ such that $P=u\circ Q$. In this case we write $P\in
\mathcal{I}\circ\mathcal{P}(^n E;F)$.\\
(c) (Linearization method) A continuous $n$-homogeneous polynomial  $P\in \mathcal{P}(^n E;F)$
  is said to be of type  $[\mathcal{I}]$ if the linear operator \[{\bar P}: E \rightarrow {\cal P}(^{n-1}E;F)~,~{\bar P}(x)(y) = {\check P}(x,y, \ldots, y)\]
belongs to ${\cal I}$. In this case we write $P \in {\cal P}_{[\mathcal{I}]}(^nE;F)$.\\
\indent It is well known that $\mathcal{L}[\mathcal{I}]$, $\mathcal{I}\circ \mathcal{P}$ and $[\mathcal{I}]$ are polynomial ideals.
\end{definition}

Given a polynomial $P \in {\cal P}(^nE;F)$, by $\check P$ we mean the (unique) continuous symmetric $n$-linear mapping from $E^n$ to $F$ such that $P(x) = \check P(x, \ldots, x)$ for every $x \in E$.
\begin{theorem}\label{polideal}
Let $\mathcal{I}$ be an operator ideal. The following are equivalent for a Banach space $E$:\\
{\rm (a)} $E$ has the $\mathcal{I}$-approximation property.\\
{\rm (b)} $\mathcal{P}(^nE;F
)=\overline{\mathcal{P}_{\mathcal{L}[\mathcal{I}]}(^n
E;F)}^{\tau_c}$ for every $n\in \mathbb{N}$ and every Banach
space $F$.\\
{\rm (c)} $\mathcal{P}(^nE;F
)=\overline{\mathcal{P}_{\mathcal{L}[\mathcal{I}]}(^n
E;F)}^{\tau_c}$ for some $n\in \mathbb{N}$ and every Banach
space $F$.\\
{\rm (d)} $\mathcal{P}(^nF;E)=\overline{\mathcal{I}\circ\mathcal{P}(^nF;E)}^{\tau_c}$
for every $n\in \mathbb{N}$ and every Banach space $F$.
{\rm (e)} $\mathcal{P}(^nF;E)=\overline{\mathcal{I}\circ\mathcal{P}(^nF;E)}^{\tau_c}$
for some $n\in \mathbb{N}$ and every Banach space $F$.\\
\indent Furthermore, if $\mathcal{L}[\mathcal{I}]
\subseteq \mathcal{I}\circ \mathcal{L}$, then the conditions above are also equivalent to:
\\
{\rm (f)} $\mathcal{P}(^nE;F)=\overline{\mathcal{I}\circ\mathcal{P}(^nE;F)}^{\tau_c}$ for every $n\in \mathbb{N}$ and every Banach space $F$.\\
{\rm (g)} $\mathcal{P}(^nE;F)=\overline{\mathcal{I}\circ\mathcal{P}(^nE;F)}^{\tau_c}$ for some $n\in \mathbb{N}$ and every Banach space $F$.
\end{theorem}

\begin{proof} (a) $\Longrightarrow$ (b) Let $P\in \mathcal{P}(^n E;F)$, $K$ be a compact subset
of $E$ and $\varepsilon >0$. Since $P$ is uniformly continuous on
$K$, there is $\delta>0$ such that $\|P(y)-P(x)\|<\varepsilon $
whenever $\|y-x\|<\delta$, $x\in K$ and $y\in E$. By the ${\cal I}$-AP of $E$ there is an operator $T
\in \mathcal{I}(E;E)$ such that $\|T(x)-x\|< \delta $ for every
$x\in K$. It follow that $\|P(T(x))-P(x)\|< \varepsilon $ for every $x\in K$. But $P\circ
T\in \mathcal{P}_{\mathcal{L}[\mathcal{I}]}(^n E;F)$, so we have that
$P\in \overline{\mathcal{P}_{\mathcal{L}[\mathcal{I}]}(^n
E;F)}^{\tau_c}$.

\medskip

\noindent (c) $\Longrightarrow$ (a) Let $u\in \mathcal{L}(E;F)$, $K$ be
a compact subset of $E$ and $\varepsilon >0$. Let $\varphi\in
E'$, $\varphi\neq 0$, and $a\in K$ be such that $\varphi(a)=1$.
Define $P\in \mathcal{P}(^nE;F) $ by $P(x)=\varphi(x)^{n-1}u(x)$. Since
$K_1:=\displaystyle\bigcup_{\varepsilon_i=\pm
1}(\varepsilon_1K+\varepsilon_2K+\cdots+\varepsilon_nK)$ is a compact
subset of $E$, by assumption there is a polynomial $Q\in
\mathcal{P}_{\mathcal{L}[\mathcal{I}]}(^n E;F)$ such that
$\|Q(x)-P(x)\|<\frac{n!\varepsilon}{n}$ for every $x\in K_1$. By the polarization formula, for every $(x_1, x_2, \ldots ,
x_n)\in K\times K\cdots \times K$ we have
$$
 \|\check{Q}(x_1,
x_2, \ldots , x_n)-\check{P}(x_1, x_2, \ldots , x_n)  \| \hspace*{15em}$$
$$=
\left\|\dfrac{1}{n!2^n}\sum_{\varepsilon_i=\pm 1}
\varepsilon_1\varepsilon_2\cdots \varepsilon_n \left[Q\left(\sum_{i=1}^{n}\varepsilon_i
x_i\right)- P\left(\sum_{i=1}^{n}\varepsilon_i x_i\right)\right]\right\|
<\dfrac{\varepsilon}{n}.
$$
From
$$\check{P}(x, a, \ldots , a)=
\dfrac{1}{n}u(x)+\dfrac{(n-1)}{n}\varphi(x)u(a),$$ it results that
$$\|n\check{Q}(x, a, \ldots , a)-u(x)-(n-1)\varphi(x)u(a) \|\hspace*{15em}
$$
$$ = n\left\|\check{Q}(x, a, \ldots ,
a)-\left(\dfrac{1}{n}u(x)+\dfrac{(n-1)}{n}\varphi(x)u(a)\right)\right\|
<\varepsilon$$
 for every $x\in K$. Considering $S=n\check{Q}(\cdot, a, \ldots ,
a)-(n-1)\varphi(\cdot)u(a)\in \mathcal{L}(E;F)$, we have $\|S(x)-u(x) \| <\varepsilon$
 for every $x\in K$. Let us check that $S\in \mathcal{I}(E;F)$. Indeed, as $Q\in
\mathcal{P}_{\mathcal{L}[\mathcal{I}]}(^n E;F)$, there are a Banach space $G$, an operator $v\in \mathcal{I}(E;G)$ and a polynomial $R\in \mathcal{P}(^n G;F)$ such that $Q=R\circ v$. Define
$T\colon G\longrightarrow F$ by $T(y)=\check{R}(y, v(a), \ldots ,
v(a))$. Then $T\circ v
\in \mathcal{I}(E;F)$ and
$$T\circ v(x) = T(v(x))=\check{R}(v(x), v(a), \ldots , v(a))=
 \check{Q}(x, a, \ldots , a),$$
 for every $x \in E$, proving that $\check{Q}(\cdot, a, \ldots , a)\in \mathcal{I}(E;F)$. On the other hand, the operator
$\varphi(\cdot)u(a)\in \mathcal{I}(E;F)$ as it is a finite rank
operator. Thus $S\in \mathcal{I}(E;F)$ and
$\mathcal{L}(E;F)=\overline{\mathcal{I}(E;F)}^{\tau_c}$. Calling on
Proposition \ref{charac} we have that $E$ has $\cal{I}$-AP.

\medskip

\noindent (a) $\Longrightarrow$ (d) Let $P\in \mathcal{P}(^n F;E)$, $K$ be a compact subset of $E$ and $\varepsilon >0$.
Since $P(K)$ is a compact subset of $E$ and $E$ has the $\mathcal{I}$-approximation property,
there is an operator $T\in \mathcal{I}(E;E)$ such that $\|T(z)-z\|< \varepsilon$ for every $z\in P(K)$.
 Hence
$\|T(P(x))-P(x)\|< \varepsilon $ for every $x\in K$. Since $T\circ P\in \mathcal{I}\circ\mathcal{P}(^n
F;E)$ we have that $P\in \overline{\mathcal{I}\circ\mathcal{P}(^n
F;E)}^{\tau_c}$.

\medskip

\noindent (e) $\Longrightarrow$ (a) The same argument of (c) $\Longrightarrow$ (a), {\it mutatis mutandis}, works in this case. We just sketch the proof: given an operator $u\in \mathcal{L}(F;E)$,  a compact set $K \subseteq F$ and $\varepsilon >0$, take $\varphi\in
F'$, $\varphi\neq 0$, and $a\in K$ such that $\varphi(a)=1$.
Defining $P=\varphi(\cdot)^{n-1}u(\cdot)\in \mathcal{P}(^nF;E) $ and a compact subset $K_1$ of $F$ as before, by assumption there is a polynomial $Q\in \mathcal{I}\circ \mathcal{P}(^n
F;E)$ such that $\|Q(x)-P(x)\|<\frac{n!\varepsilon}{n}$ for every
$x\in K_1$. Define $S=n\check{Q}(\cdot, a, \ldots ,
a)-(n-1)\varphi(\cdot)u(a) \in {\cal L}(F;E)$ and proceed exactly as above to get $\|S(x)-u(x) \| <\varepsilon$
 for every $x\in K$. Write $Q=v\circ R$ with
$v\in \mathcal{I}(G;E)$ and $R\in \mathcal{P}(^n F;G)$ and define $T \in {\cal L}(F;G)$ by $T(y)=\check{R}(y, a, \ldots ,
a)$. Thus $v\circ T = \check Q(\cdot, a, \ldots, a) \in \mathcal{I}(F;E)$ and this implies that $S\in \mathcal{I}(F;E)$.

\medskip

\noindent Since (b) $\Longrightarrow$ (c) and (d) $\Longrightarrow$ (e) are obvious and (b) $\Longrightarrow$ (a) and (c) $\Longrightarrow$ (a) follow by taking $n = 1$ in Proposition \ref{charac}, the first part of the proof is complete.

Assume now that $\mathcal{L}[\mathcal{I}]
\subseteq \mathcal{I}\circ \mathcal{L}$.\\
(a) $\Longrightarrow$ (f) $E$ has $\mathcal{I}$-AP by assumption. Let $n\in \mathbb{N}$, $P\in \mathcal{P}(^n
E;F)$, $K$ be a compact subset of $E$ and $\varepsilon >0$. Note
that $P= P_L\circ \sigma_n$ where $\sigma_n\in
\mathcal{P}(^n E; \hat\otimes_\pi^{n,s}E)$ is the canonical $n$-homogeneous polynomial defined by $\sigma_n(x)=
x\otimes \cdots \otimes x$ and $P_L\in
\mathcal{L}(\hat\otimes_\pi^{n,s}E; F)$ is the linearization of $P$, that is $P_L(x\otimes \cdots \otimes x)=P(x)$. By
Corollary \ref{cortens} we have that $\hat\otimes_\pi^{n,s}E$ has $\mathcal{I}$-AP, hence $\mathcal{L}(\hat\otimes_\pi^{n,s}E;
F)=\overline{\mathcal{I}(\hat\otimes_\pi^{n,s}E; F)}^{\tau_c}$ by Proposition \ref{charac}. So for the compact subset
$\sigma(K)$ of $\hat\otimes_\pi^{n,s}E$
there is an operator $u\in \mathcal{I}(\hat\otimes_\pi^{n,s}E; F)$ such that
$$\|u \circ \sigma_n(x) - P(x)\| = \|u(\sigma_{n}(x))-P_L(\sigma_n(x)) \|<\varepsilon$$
for every $x\in K$. Since $Q=u\circ \sigma_{n}\in \mathcal{I}\circ\mathcal{P}(^nE;F)$ we
have that $P\in
\overline{\mathcal{I}\circ\mathcal{P}(^nE;F)}^{\tau_c} $.\\
(f) $\Longrightarrow$ (g)  is obvious and (g) $\Longrightarrow$ (a) follows from a repetition of the argument of the proofs of (c) $\Longrightarrow$ (a) and (e) $\Longrightarrow$ (a), therefore the proof is complete.
\end{proof}

The spaces $\mathcal{P}_{\mathcal{L}[\mathcal{I}]}(^n
E;E)$ and $\mathcal{I}\circ\mathcal{P}(^nE;E)$ are often different. We have obtained situations where, even though different, their $\tau_c$-closures  coincide:

\begin{corollary}\label{coincidence}
Let $\mathcal{I}$ be an operator ideal.\\
{\rm (a)} If the Banach spaces $E$ and $F$ have the $\mathcal{I}$-approximation property, then $\overline{\mathcal{P}_{\mathcal{L}[\mathcal{I}]}(^n
E;F)}^{\tau_c}=\mathcal{P}(^nE;F
)=\overline{\mathcal{I}\circ\mathcal{P}(^n F;E)}^{\tau_c}$ for every $n\in \mathbb{N}$.\\
{\rm (b)} A Banach space $E$ has the $\mathcal{I}$-approximation property if and only if $\overline{\mathcal{P}_{\mathcal{L}[\mathcal{I}]}(^n
E;E)}^{\tau_c}=\mathcal{P}(^nE;E
)=\overline{\mathcal{I}\circ\mathcal{P}(^n E;E)}^{\tau_c}$ for every $n\in \mathbb{N}$.
\end{corollary}

\begin{example}\rm It is not difficult to check that neither $\mathcal{P}_{\mathcal{L}[\mathcal{W}]}(^2 \ell_1; \ell_1) \subseteq \mathcal{W} \circ \mathcal{P} (^2 \ell_1; \ell_1)$ nor $\mathcal{W} \circ \mathcal{P} (^2 \ell_1; \ell_1) \subseteq  \mathcal{P}_{\mathcal{L}[\mathcal{W}]}(^2 \ell_1; \ell_1)$ (see \cite[Examples 27 and 28]{Botelho1}). Nevertheless, by Corollary \ref{coincidence}(b) both subspaces are $\tau_c$-dense in $\mathcal{P}(^2 \ell_1;\ell_1)$ because $\ell_1$ has the approximation property (hence has the weakly compact approximation property).
\end{example}

The following result appears in \cite{Caliskan4}:

\begin{theorem}\label{errado}{\rm \cite[Theorem 11]{Caliskan4}} The following are equivalent for a Banach space $E$:\\
{\rm (a)} $E$ has the weakly compact approximation property.\\
{\rm (b)} $\mathcal{P}(^nE;F
)=\overline{\mathcal{P}_{[\mathcal{W}]}(^n
E;F)}^{\tau_c}$ for every $n\in \mathbb{N}$ and every Banach
space $F$.\\
{\rm (c)} $\mathcal{P}(^nE;F
)=\overline{\mathcal{P}_{[\mathcal{W}]}(^n
E;F)}^{\tau_c}$ for some $n\in \mathbb{N}$ and every Banach
space $F$.
\end{theorem}

Unfortunately there is a gap in the proof of this theorem (see the MathSciNet review of this paper by C. Boyd). In this direction we have:

\begin{proposition}\label{certo} Let $\cal I$ be a closed injective operator ideal. The following are equivalent for a Banach space $E$:\\
{\rm (a)} $E$ has the $\cal I$-approximation property.\\
{\rm (b)} $\mathcal{P}(^nE;F
)=\overline{\mathcal{P}_{[\mathcal{I}]}(^n
E;F)}^{\tau_c}$ for every $n\in \mathbb{N}$ and every Banach
space $F$.\\
{\rm (c)} $\mathcal{P}(^nE;F
)=\overline{\mathcal{P}_{[\mathcal{I}]}(^n
E;F)}^{\tau_c}$ for some $n\in \mathbb{N}$ and every Banach
space $F$.
\end{proposition}

\begin{proof} Just combine Theorem \ref{polideal} with the fact that $[\mathcal{I}] = {\cal L}[\mathcal{I}]$ whenever the operator ideal $\cal I$ is closed and injective (see \cite{braunss-junek}).
\end{proof}

Recalling that $\cal W$ is closed and injective, Proposition \ref{certo} fixes Theorem \ref{errado} and generalizes it to arbitrary closed injective operator ideals.

\section{Spaces of holomorphic functions}
The approximation property and its variants in spaces of holomorphic functions and their preduals have been largely investigated (see, e.g., \cite{as, Rueda, Caliskan3, Caliskan1, dm1, dm2, mu}). In this section we study the $\cal I$-approximation property in spaces of holomorphic functions of bounded type, spaces of weakly uniformly continuous holomorphic functions, spaces of bounded holomorphic functions and/or their preduals. For background on infinite-dimensional holomorphy we refer to \cite{Dineen, Mujica}. An important issue of this section is the combination of results from different sections of the paper.

\indent All spaces in this section are supposed to be complex.\\
\indent Spaces of holomorphic functions, spaces of bounded holomorphic functions and spaces of weakly uniformly continuous holomorphic functions, as well as their corresponding preduals, are locally convex spaces, so we have to say a few words about the definition of the $\cal I$-approximation property in the setting of locally convex spaces. The definition of operator ideals (on Banach spaces) can be naturally generalized to the concept of operator ideals on locally convex spaces (details can be found in \cite[Chapter 29]{Pietsch}). We say that an operator ideal $\cal U$ on locally convex spaces is an extension of an operator ideal $\cal I$ on Banach spaces if $\mathcal{U}(E;F)=\mathcal{I}(E;F)$ for all Banach spaces $E$
and $F$. There are several ways to extend an operator ideal on Banach spaces to an operator ideal on locally convex spaces (see \cite[Section 29.5]{Pietsch}). In this paper we shall work with the smallest of such natural extensions, which we describe next. Given an operator ideal $\cal I$ on Banach spaces, an operator $S\in \mathcal{L}(U;V)$ between locally convex spaces belongs to the {\it inferior extension of $\mathcal{I}$} if there exist Banach spaces $E$ and $F$ and operators $A\in \mathcal{L}(U, E)$, $T\in \mathcal{I}(E, F)$ and $Y\in \mathcal{L}(F, V)$ such that $S=Y\circ T\circ A$. In this case, for the sake of simplicity, we still write $S \in \mathcal{I}(U;V)$. Of course we can consider the compact-open topology on ${\cal L}(U;U)$ for a locally convex space $U$, so Definition \ref{def} makes sense for an operator ideal $\cal I$ on Banach spaces and a locally convex space $U$, hence the $\cal I$-approximation property is well defined for locally convex spaces.\\
\indent Unless explicitly stated otherwise, an operator ideal means an operator ideal on Banach spaces and an statement like ${\cal I}_1 \subseteq {\cal I}_2$ means that ${\cal I}_1(E;F) \subseteq {\cal I}_2(E;F)$ for all Banach spaces $E$ and $F$.

\begin{remark}\label{remark}\rm It is easy to see that Propositions \ref{propos}, \ref{complemented} and \ref{sum} hold true in the realm of locally convex spaces. Of course, in condition (d) of Proposition \ref{propos}, $\|T(x)-x\|$ is replaced by $p(T(x) -x)$ where $p$ is an arbitrary continuous semi-norm. The proof of the locally convex version of Proposition \ref{propos} follows the lines of the proof of \cite[43(1)]{kothe}.
\end{remark}

\begin{definition}\rm  A sequence $\{E_n\}_{n=1}^{\infty}$ of subspaces of a locally convex space $E$ is said to be a {\it decomposition of $E$} if any $x \in E$ can be written in a unique way as $x = \sum_{n=1}^\infty x_n$ with $x_n \in E_n$ for every $n$ and the projection $\sum_{n=1}^\infty x_n \mapsto \sum_{n=1}^m x_n$ is continuous for every $m \in \mathbb{N}$.\\
\indent  Let $\mathcal{S}=\{(\alpha_n)_{n=1}^{\infty}: \alpha_n\in \mathbb{C}
\,\,\,\text{and}
\,\,\,\limsup_{n\rightarrow\infty}|\alpha_n|^{\frac{1}{n}}\leq 1
\} $. A decomposition $\{E_n\}_{n=1}^{\infty}$ of $E$ is an {\it $\mathcal{S}$-absolute decomposition} if
\begin{enumerate}
\item[(1)] $\sum_{n=1}^{\infty} x_n\in E, \, x_n\in E_n$ for all
$n$ and $(\alpha_n)_{n=1}^{\infty}\in \mathcal{S}$ implies
$\sum_{n=1}^{\infty}\alpha_nx_n\in E,$

\item[(2)] If $p$ is a continuous semi-norm on $E$ and
$(\alpha_n)_{n=1}^{\infty}\in \mathcal{S}$ then
$$p_{\alpha}\left(\sum_{n=1}^{\infty}x_n
\right):=\sum_{n=1}^{\infty}|\alpha_n|p(x_n)$$ defines a continuous semi-norm
on $E$. \label{condicao2}
 \end{enumerate}
Further details can be found in \cite[Section 3.3]{Dineen}.
\end{definition}


\begin{lemma}\label{decomposition}
Let $\mathcal{I}$ be an operator ideal.
If $\{E_n\}_{n=1}^{\infty}$ is an  $\mathcal{S}$-absolute decomposition
of the locally convex space $E$, then $E$ has the $\mathcal{I}$-approximation property if and only if each $E_n$
has the  $\mathcal{I}$-approximation property.
\end{lemma}

\begin{proof} An adaptation of the proof of \cite[Proposition 1]{Rueda} works in this case. We give the details for the sake of completeness. Suppose that each $E_n$ has the $\mathcal{I}$-approximation
property. Let $K$ denote a compact subset of  $E$, let  $p$ denote
an arbitrary continuous semi-norm on  $E$ satisfying condition
$(2)$ above with $\alpha_n=1$ for all $n$, and let $\varepsilon>0$ be
arbitrary.

Define  $$\pi_{n}\left(\sum_{m=1}^{\infty} x_m \right):=\sum_{m=1}^{n}x_m$$ for all $\sum_{m=1}^{\infty} x_m\in E$ and let $\pi^n=id_E-\pi_n$.
By \cite[Lemma 3.33]{Dineen} there exists a positive integer  $n_0$ such that
 $$\sup\{ p(\pi^{n_0}(x)): x\in K\}<\varepsilon.   $$
Consider  $F_{n_0}:=\pi_{n_0}(E)=\sum_{n=1}^{n_0}E_n$. By the locally convex version of Proposition~\ref{sum} (see Remark \ref{remark}) $F_{n_0}$ has the
$\mathcal{I}$-approximation property. Since  $\pi_{n_0}(K)$ is a
compact subset of  $F_{n_0}$ there exists an operator
$T\in\mathcal{I}(F_{n_0};F_{n_0})$ such that  $$
p(\pi_{n_0}(x)-T(\pi_{n_0}(x)))<\varepsilon   $$  for every $x\in K$.
Using the natural inclusion $i \colon F_{n_0}\hookrightarrow E$ we see that
$R:=i\circ T\circ \pi_{n_0} \in \mathcal{I}(E;E)$. Hence
\begin{eqnarray*}
p(x-R(x))&\leq&p(x-\pi_{n_0}(x))+p(\pi_{n_0}(x)-T(\pi_{n_0}(x))\\
&=&p(\pi^{n_0}(x))+p(\pi_{n_0}(x)-T(\pi_{n_0}(x))<2\varepsilon
\end{eqnarray*}
for every $x\in K$. So $id_E \in \overline{{\cal I}(E;E)}^{\tau_c}$. It follows that $E$ has $\cal I$-AP by the locally convex version of Proposition \ref{propos}.  \\
\indent Conversely, since each $E_n$ is a complemented subspace of $E$ and $E$ has the $\mathcal{I}$-approximation property, by the locally convex version of Proposition~\ref{complemented} it follows that each $E_n$ has $\mathcal{I}$-approximation property as well.
\end{proof}

By $\mathcal{P}_{w}(^nE)$ we mean the closed subspace of $\mathcal{P}(^nE)$ of all continuous $n$-homogeneous polynomials that are weakly continuous on bounded sets. Let $U$ be an open subset of a Banach space $E$. A bounded subset $A$ of $U$ is {\it $U$-bounded} if there is a 0-neighborhood $V$ such that $A+V \subseteq U$. By $\mathcal{H}_b(U;F)$ we denote the space of holomorphic functions $f \colon U \longrightarrow F$, where $F$ is a Banach space, of {\it bounded type}, that is, $f$ is bounded on $U$-bounded sets. If $F = \mathbb{C}$ we simply write $\mathcal{H}_b(U)$. The symbol $\mathcal{H}_{wu}(U)$ stands for the space of all holomorphic functions $f \colon U \longrightarrow \mathbb{C}$ that are weakly uniformly continuous on $U$-bounded sets. When endowed with the topology of uniform convergence on $U$-bounded sets, both $\mathcal{H}_b(U;F)$ and $\mathcal{H}_{wu}(U)$ are locally convex spaces.

\begin{proposition}\label{holomor} Let $\mathcal{I}$ be an operator ideal, $U$ be a balanced open subset of the Banach space $E$ and $F$ be a Banach space.\\
{\rm (a)} $\mathcal{H}_b(U;F)$ has $\mathcal{I}$-$AP$ if and only if $\mathcal{P}(^n E; F)$ has $\mathcal{I}$-$AP$ for every $n\in \mathbb{N}$.\\
{\rm(b)} $\mathcal{H}_{wu}(U)$ has $\mathcal{I}$-$AP$ if and only if $\mathcal{P}_{w}(^nE)$ has $\mathcal{I}$-$AP$ for every $n\in \mathbb{N}$.
\end{proposition}

\begin{proof} Just combine Lemma \ref{decomposition} with the facts that $\{\mathcal{P}(^n E; F)\}_{n=1}^\infty$ is an $\mathcal{S}$-absolute decomposition of $\mathcal{H}_b(U;F)$ (this follows from an adaptation of the proof of \cite[Proposition 3.36]{Dineen}) and that $\{\mathcal{P}_w(^n E)\}_{n=1}^\infty$ is an $\mathcal{S}$-absolute decomposition of $\mathcal{H}_{wu}(U;F)$ (see the proof of \cite[Theorem 9]{Rueda}).
\end{proof}

In the sequel some of our apparently disconnected results will be combined altogether. A Banach space $E$ is said to be {\it polynomially reflexive} if ${\cal P}(^nE)$ is reflexive for every $n\in \mathbb{N}$. For example, Tsirelson's original space $T^*$ is polynomially reflexive \cite{aad}.
\begin{proposition}\label{unification} Let $\cal I$ be an operator ideal such that $\mathcal{L}[\mathcal{I}] \subseteq \mathcal{I}\circ\mathcal{L}$ and ${\cal I} \subseteq {\cal I}^{dual}$ or ${\cal I}^{dual} \subseteq {\cal I}$. The following are equivalent for a polynomially reflexive Banach space $E$ and a balanced open subset $U$ of $E$:\\
{\rm (a)} $E$ has $\cal I$-AP.\\
{\rm (b)} ${\cal P}(^nE)$ has $\cal I$-AP for every $n\in \mathbb{N}$.\\
{\rm (c)} ${\cal P}(^nE)$ has $\cal I$-AP for some $n\in \mathbb{N}$.\\
{\rm (d)} $\mathcal{H}_b(U)$ has $\cal I$-AP.
\end{proposition}
\begin{proof} (a) $\Longrightarrow$ (b) Let $n \in \mathbb{N}$. By Corollary \ref{cortens} we know that $\hat\otimes_\pi^{n,s}E$ has $\cal I$-AP. Since ${\cal P}(^nE)$ is isomorphic to $\left(\hat\otimes_\pi^{n,s}E\right)'$ and these spaces are reflexive, by Corollary \ref{dualIAP} we have that ${\cal P}(^nE)$ has $\cal I$-AP.\\
(b) $\Longrightarrow$ (c) This implication is obvious.\\
(c) $\Longrightarrow$ (a) By \cite[Proposition 5.3]{as} it follows that $E'$ is isomorphic to a complemented subspace of ${\cal P}(^nE)$, so $E'$ has $\cal I$-AP by Proposition \ref{complemented}. Then $E$ has $\cal I$-AP by Corollary \ref{dualIAP}.\\
(d) $\Longleftrightarrow$ (b) This equivalence follows from Corollary \ref{holomor}(a).
\end{proof}

To get another connection of the results from different sections we consider the predual of the space of holomorphic functions: given an open subset $U$ of a Banach space, Mazet \cite{mazet} proved the existence of a complete locally convex space $G(U)$ and of a canonical holomorphic function $\delta_U \colon U \longrightarrow G(U)$ such that for every Banach space $F$ and every holomorphic function from $U$ to $F$ there is a unique continuous linear operator $T_f$ from $G(U)$ to $F$ such that $f = T_f \circ \delta_U$.

\begin{proposition}\label{predual} Let $U$ be a balanced open subset of the Banach space $E$ and $\mathcal{I}$ be an operator ideal such that $\mathcal{L}[\mathcal{I}] \subseteq \mathcal{I}\circ\mathcal{L}$. Then $E$ has $\cal I$-AP if and only if $G(U)$ has $\cal I$-AP.
\end{proposition}

\begin{proof} For every $n \in \mathbb{N}$ let $Q(^nE)$ be the space of all linear functionals $\varphi$ on ${\cal P}(^nE)$ such that the restriction of $\varphi$ to each locally bounded subset is ${\tau}_c$-continuous. By \cite[Proposition 4]{boyd} we have that $\{Q(^nE)\}_{n=1}^\infty$ is an $\cal S$-absolute decomposition of $G(U)$, so, by Lemma \ref{decomposition}, $G(U)$ has $\cal I$-AP if and only if $Q(^nE)$ has $\cal I$-AP for every $n$. But $Q(^nE)$ is isomorphic to $\hat\otimes_\pi^{n,s}E$ (see \cite[p.\,223]{boyd}), so by Corollary \ref{cortens} we have that $G(U)$ has $\cal I$-AP if and only if $Q(^nE)$ has $\cal I$-AP for every $n$ if and only if $\hat\otimes_\pi^{n,s}E$ has $\cal I$-AP for every $n$ if and only if $E$ has $\cal I$-AP.
\end{proof}

The results from Section 6 have not been combined with results from other sections yet. For results of Section 6 to come into play we investigate the $\cal I$-approximation property in the predual of the space ${\cal H}^\infty(U;F)$ of bounded holomorphic functions from an open subset $U$ of a Banach space $E$ to a Banach space $F$. ${\cal H}^\infty(U;F)$ is a Banach space with the sup norm. Let $U$ be an open subset of a Banach space $E$. Mujica \cite{Mujica1} proved the existence of a Banach space $G^{\infty}(U)$ and of a canonical bounded holomorphic mapping $\delta_{U}\in {\cal{H}}^{\infty}(U; G^{\infty}(U))$ with the following universal property:
 to every $f\in\mathcal{H}^{\infty}(U; F)$ corresponds a unique
 linear operator $T_f \in \mathcal{L}(G^{\infty}(U); F)$ such that
 $f=T_f\circ \delta_U$. He also introduced a very useful locally convex topology on ${\cal H}^\infty(U;F)$:

\begin{theorem}{\rm \cite[Theorem 4.8]{Mujica1}}\label{teomujica}
Let $E$ and $F$ be Banach spaces, and let $U$ be an open subset of
$E$. Let $\tau_{\gamma}$ denote the locally convex topology on
${\cal H}^{\infty}(U; F)$ generated by the seminorms of the form
$$p(f)=\sup_{j} \alpha_j\|f(x_j)\|,$$
where $(x_j)$ varies over all sequences in $U$ and $(\alpha_j)$
varies over all sequences of positive real numbers tending to zero. Then
the mapping
$$f\in (\mathcal{H^{\infty}}(U; F), \tau_{\gamma}) \longrightarrow T_f \in (\mathcal{L}(G^{\infty}(U), \tau_c)$$
is a topological isomorphism.
\end{theorem}
We denote by $\mathcal{I}\circ \mathcal{H}^{\infty}(U; F)$ the
collection of all $f \in \mathcal{H}^{\infty}(U;F)$ so that
$f=u\circ g$, where $G$  is a Banach space, $g\in
\mathcal{H}^{\infty}(U;G)$ and $u\in \mathcal{I}(G;F)$. Next result extends \cite[Theorem 5]{Caliskan1}.

\begin{theorem} Let $\mathcal{I}$ be an operator ideal such that $\cal{L}[\cal{I}]\subseteq \cal{I}\circ
\cal{L}$. The following conditions are
equivalent for a Banach space $E$ and an open subset $U$ of $E$:\\
{\rm(a)} $ E$ has $\mathcal{I}$-$AP$.\\
{\rm (b)} $\mathcal{H}^{\infty}(U;F)= \overline{\cal{I}\circ
\cal{H}^{\infty}(U; F)}^{\tau_{\gamma}}$ for every
Banach space $F$.\\
{\rm (c)} $G^{\infty}(U)$ has $\cal{I}$-AP.
\end{theorem}

\begin{proof} (a) $\Longrightarrow$ (b) Let $f\in
\mathcal{H^{\infty}}(U;F)$. Let $p$ be a continuous semi-norm on $
(\mathcal{H^{\infty}}(U;F), \tau_{\gamma})$. By \cite[Proposition
5.2]{Mujica1} there are homogeneous polynomials $P_j\in P(^jE;F)$, $j=0, 1, \ldots,n$, such that
$p(P-f)<\frac{\varepsilon}{2}$ where $P = P_0 +P_1+ \cdots + P_n$. Since $E$
has $\cal{I}$-AP and $\cal{L}[\cal{I}]\subseteq \cal{I}\circ
\cal{L}$, it follows from Proposition \ref{polideal} that
$\mathcal{P}(^jE;F)=\overline{\mathcal{I}\circ\mathcal{P}(^jE;F)}^{\tau_c}$
for every $j\in \mathbb{N}$. On the other hand, by \cite[Proposition
4.9]{Mujica1}, $\tau_{\gamma}= \tau_{c}$ on $P\in P(^{j}E;F)$ for
every $j\in \mathbb{N}$. So there are homogeneous polynomials $Q_j\in \mathcal{I}\circ
\mathcal{P}(^j E;F)$ such that $$p(Q_j-P_j)<\frac{\varepsilon}{2(n+1)}$$
for every $j=0, 1, \ldots n$. Putting $Q=Q_0+Q_1+ \cdots +Q_n$, mimicking an
argument used in the proof of \cite[Theorem 2.2]{Botelho2} one can easily prove that $Q$ is
of the form $Q=u\circ R$ where $u\in {\cal I}(E;G)$, $G$ is a Banach space and $R$ is a finite sum of homogeneous polynomials from $G$ to $F$.
Then the restriction of $Q$ to $U$, still denoted by $Q$, is a bounded holomorphic function, so $Q \in \mathcal{I}\circ \mathcal{H^{\infty}}(U; F)$. Since
$$p(Q-P) = p\left(\sum_{j=0}^n Q_j - \sum_{j=0}^n P_j  \right) \leq \sum_{j=0}^n p(Q_j - P_j) <\frac{\varepsilon}{2},$$
if follows that
$$p(Q-f)\leq
p(Q-P)+p(P-f)< \frac{\varepsilon}{2} + \frac{\varepsilon}{2} = \varepsilon,$$
which proves (b).

\medskip

\noindent (b) $\Longrightarrow$ (c) By \cite[Theorem
2.1]{Mujica1}, $\delta_{U}\in
\mathcal{H}^{\infty}(U;G^{\infty}(U))$. Taking $F=G^{\infty}(U)$ in
(b), we have that $\delta_{U}\in
\overline{\mathcal{H}^{\infty}(U;G^{\infty}(U))}^{\tau_{\gamma}}$. Hence there is a net $(f_\alpha)\subseteq
\mathcal{H}^{\infty}(U;G^{\infty}(U))$ such that $\displaystyle
f_{\alpha}\stackrel{\tau_{\gamma}}\longrightarrow \delta_{U}$. As to the corresponding net $(T_{f_\alpha})$ of linear operators, by Theorem \ref{teomujica} we get that
$$T_{f_\alpha} \stackrel{\tau_c}{\longrightarrow} T_{\delta_U} = id_{G^\infty(U)}. $$ But
\cite[Proposition 4.2]{Botelho2} gives that $(T_{f_\alpha})\subseteq \mathcal{I}(G^{\infty}(U);
G^{\infty}(U) )$. Therefore $id_{G^{\infty}(U)}\in
\overline{\mathcal{I}(G^{\infty}(U); G^{\infty}(U) )}^{\tau_c}$. By
Proposition \ref{charac} we have that $G^{\infty}(U)$ has $\mathcal{I}-$AP.

\medskip

\noindent (c) $\Longrightarrow$ (a) By \cite[Proposition
2.3]{Mujica1}, $E$ is topologically isomorphic to a complemented
subspace of $G^{\infty}(U)$, which has
$\mathcal{I}$-AP by assumption. It follows from Proposition \ref{complemented}
that $E$ has $\mathcal{I}-$AP.
\end{proof}

\vspace*{1em} \noindent Faculdade de Matem\'atica\\
Universidade Federal de Uberl\^andia\\
38.400-902 - Uberl\^andia - Brazil\\
e-mails: soniles@famat.ufu.br, botelho@ufu.br

\end{document}